\documentclass[12pt]{amsart}
\usepackage{amsfonts,latexsym,rawfonts,amsmath,amssymb,amsthm,braket,cite}

\newtheorem{theorem}{Theorem}[section]
\newtheorem{lemma}[theorem]{Lemma}
\newtheorem{proposition}[theorem]{Proposition}

 \theoremstyle{definition}
\newtheorem{definition}[theorem]{Definition}

\theoremstyle{remark}
\newtheorem{remark}[theorem]{Remark}

\numberwithin{equation}{section}

\makeatletter
\@namedef{subjclassname@2020}{\textup{2020} Mathematics Subject Classification}
\makeatother

\begin{document}

\title[$k$-convex hypersurface for prescribed curvature]{The existence of $k$-convex hypersurface for a class of Hessian  curvature equations}

\author{Kang Xiao}
\address{Faculty of Mathematics and Statistics, Hubei Key Laboratory of Applied Mathematics, Hubei University,  Wuhan 430062, P.R. China}
\email{dalangbeikang@163.com}

\author{Jiabao Gong$^{\ast}$}
\address{Faculty of Mathematics and Statistics, Hubei Key Laboratory of Applied Mathematics, Hubei University,  Wuhan 430062, P.R. China}
\email{202321104011284@stu.hubu.edu.cn}

\thanks{$\ast$ Corresponding author}



\begin{abstract}
This article investigates the existence of closed, star-shaped hypersurfaces for a class of Hessian  curvature equations. By combining a priori estimates with the continuity method, we establish the existence and uniqueness of $k$-convex hypersurfaces for both nonhomogeneous and homogeneous Hessian curvature equations, and by establishing a constant rank theorem, we prove that the resulting $k$-convex hypersurfaces are strictly convex.
\end{abstract}

\maketitle

\baselineskip18pt
\section{Introduction}

The classical prescribed Weingarten curvature problem refers to finding a hypersurface $M\subset\mathbb{R}^{n+1}$
whose principal curvatures $\kappa$ satisfy
\begin{equation*}
\sigma_k(\kappa(X))=f(X),\quad X \in M,
\end{equation*}
where $f$ is a given function and $\sigma_k$ denotes the $k$-th elementary symmetric function. This problem has attracted considerable attention.
In the mean curvature case \(k=1\), related results were established by Bakelman--Kantor \cite{Ba1,Ba2} and Treibergs--Wei \cite{Tr}, while the Gauss curvature case \(k=n\) was studied by Oliker \cite{Ol}. More general prescribed Weingarten curvature equations were considered by Aleksandrov \cite{Al1}, Firey \cite{Fi}, and Caffarelli--Nirenberg--Spruck \cite{Ca}. These results have also been extended to various Riemannian settings, including the unit sphere by Li--Oliker \cite{Li-Ol}, space forms by Barbosa--Lira--Oliker \cite{Ba-Li}, hyperbolic spaces by Jin--Li \cite{Jin}, warped product manifolds by Andrade--Barbosa--Lira \cite{An}, and Riemannian manifolds admitting a global normal Gaussian coordinate system by Li--Sheng \cite{Li-Sh}.

More generally, the prescribed Weingarten curvature equation can be written as
\begin{equation*}
  \sigma_k(\kappa)=f(X,\nu(X)),\quad X \in M,
\end{equation*}
where $f$ depends on the position vector $X$ and the outward unit normal $\nu$. Curvature estimates are a central issue in the study of such equations. For $k=1$, these estimates follow from the theory of quasilinear elliptic equations. For $k=2$, global $C^2$ estimates were established by Guan--Ren--Wang \cite{Guan-Ren15}. Spruck--Xiao \cite{SX} extended the $2$-convex case to space forms and also provided a simple proof for hypersurfaces in Euclidean space. For $2 < k < n$, $C^2$ estimates were obtained for prescribed curvature measure equations in \cite{Guan12,Guan09}, where
\[
f = \widetilde{f}\left(\frac{X}{|X|}\right)|X|^{-(n+1)}\langle X,\nu\rangle .
\]
Ren--Wang \cite{Ren,Ren1} proved $C^2$ estimates for the cases $k=n-1$ and $k=n-2$. For $k=n$, curvature estimates for general functions $f(X,\nu)$ were obtained by Caffarelli--Nirenberg--Spruck \cite{Caffarelli-Nirenberg-Spruck-1984}. Ivochkina \cite{Iv1,Iv2} studied the Dirichlet problem for the above equation on domains in $\mathbb{R}^n$, obtaining $C^2$ estimates under additional assumptions on the dependence of $f$ on $\nu$. Guan--Ren--Wang \cite{Guan-Ren15} established global $C^2$ estimates for closed convex hypersurfaces in Euclidean space, and Chen--Li--Wang \cite{Chen-Li-Wang-2018} obtained curvature estimates for convex hypersurfaces in general warped product manifolds.

In this paper, we consider the following prescribed Weingarten curvature equation
\begin{equation}\label{G-eqk}
\sigma_k(\kappa)=f\!\left(\frac{X}{|X|}\right) |X|^b \langle X, \nu\rangle^q,
\quad X \in M,
\end{equation}
where $1\leq k<n$, $b, q$ are real numbers, and $f \in C^2(\mathbb{S}^n)$ is a positive function. The main purpose is to study the existence, uniqueness and strict convexity of $k$-convex star-shaped hypersurfaces satisfying \eqref{G-eqk}. Recall that $k$-convexity is defined as follows \cite{Guan-Lin-Ma-2006,HX-2013,Guan09}.

\begin{definition}
A regular $C^2$-hypersurface $M$ is called $k$-convex if, at each $X\in M$, $\kappa(X)$ satisfies $\kappa \in \Gamma_{k}$, where
$\Gamma_{k}=\{\xi \in \mathbb{R}^n: \sigma_{j}(\xi)>0,~ \forall ~ 1\leq j \leq k\}$.
\end{definition}

An $n$-convex hypersurface is strictly locally convex.  When $k=n$, looking for a strictly convex hypersurface $M=\{X\in\mathbb{R}^{n+1}:X=\rho(x)x, ~ x \in \mathbb{S}^n\}$ satisfying equation \eqref{G-eqk} is equivalent to solving the  Monge--Amp\`ere   type equation \begin{equation}\label{Lp-eq}
 \det (D^2 h + h I) = h^{-q} (h^2 + |D h|^2)^{-\frac{b}{2}} f^{-1}, \quad \text{on} \quad \mathbb{S}^n,
 \end{equation}
  where $h$ is the support function of the convex body enclosed by $M$, $I$ is the identity matrix, $Dh$ and $D^2h$ are the gradient and Hessian tensor of $h$ with respect to the standard metric on $\mathbb{S}^n$. Equation \eqref{Lp-eq} is exactly the equation considered in the $L_p$ dual Minkowski problem. For references related to this problem, see \cite{BF-2019, CHZ-2019, Chen-Li-2021, HZ-2018, LYZ-2018-1}.

Equation \eqref{G-eqk} is also related to  curvature measure problem. Let
\[
M = \{ X \in \mathbb{R}^{n+1} : X = \rho(x)x,\ x \in \mathbb{S}^n \}
\]
be a star-shaped hypersurface satisfying \eqref{G-eqk}. When $b = -n-1$, $q = 1$, and $1 \leq k \leq n$, equation \eqref{G-eqk} reduces to
\begin{equation}\label{curvature-measure-eq}
\sigma_k(\kappa) = \rho^{1-n}(\rho^2 + |D\rho|^2)^{-\frac{1}{2}} f,  \quad \text{on } \quad \mathbb{S}^n,
\end{equation}
which is precisely the equation associated with the prescribed $(n-k)$-th curvature measure problem. In the case $k=n$, this becomes the prescribed $0$-th curvature measure problem, namely the Aleksandrov problem. The existence and uniqueness of solutions were obtained by Aleksandrov \cite{Aleksandrov-1942}. The regularity theory in the elliptic case was developed by Pogorelov \cite{Pogorelov-1973} for $n=2$ and by Oliker \cite{Oliker-1983} in higher dimensions. The degenerate case was studied by Guan--Li \cite{Guan-Li-1997}. For general $1 \leq k \leq n$, finding $k$-convex solutions of \eqref{curvature-measure-eq} remains an interesting problem; see \cite{Guan12,HX-2013,Guan09}.
In the case $1 \leq k < n$, $q \leq 1$, and $-b - q - k > 0,$
the present paper establishes the existence and uniqueness of a $k$-convex solution to \eqref{G-eqk}. This generalizes the corresponding results in \cite{Guan12,HX-2013,Guan09}. The first main result is stated as follows.

\begin{theorem}\label{Th-One}
  Let $1\leq k <n $, $q\leq 1$, and $f \in C^2(\mathbb{S}^n)$ be a positive function. If $-b-q-k>0$, then there exists a unique  $k$-convex star-shaped hypersurface $M \in C^{3,\alpha}$ with $\alpha \in (0,1)$ such that it satisfies equation \eqref{G-eqk}. Furthermore, if $b\leq -k$ and
\begin{equation}\label{908998}
f^{-\frac{1}{k}}\left(\frac{X}{|X|}\right) |X|^{-\frac{b}{k}}~\mbox{is locally convex in}~\mathbb{R}^{n+1}  \setminus \{0\},
\end{equation}
then $M$ is strictly convex.
\end{theorem}

\begin{remark}
For equation \eqref{G-eqk} with $b = -n-1$ and $q = 1$, Guan--Lin--Ma \cite{Guan09} established the existence of a strictly convex hypersurface, while Guan--Li--Li \cite{Guan12} obtained the existence of a $k$-convex hypersurface. Huang--Xu \cite{HX-2013} proved the existence of a $k$-convex star-shaped hypersurface for \eqref{G-eqk} with $b = -n-1$, $q \leq 1$, and $q \neq 0$. Therefore, Theorem \ref{Th-One} contains the results of \cite{Guan12,Guan09,HX-2013} as special cases, since it allows a wider range of  $b$ and $q$.
\end{remark}

When $-b-q-k=0$, the following theorem gives the second main result of this paper.

\begin{theorem}\label{Th-Two}
  Let $1\leq k <n $, $q\leq 1$, and $f \in C^2(\mathbb{S}^n)$ be a positive function. If $-b-q-k=0$ and $q \neq 0$, then there exist a positive constant $\gamma$ and a $k$-convex star-shaped hypersurface $M \in C^{3,\alpha}$ with $\alpha \in (0,1)$ satisfying the homogeneous curvature equation
  \begin{equation}\label{eq:1.4}
    \sigma_k(\kappa)=\gamma f\left(\frac{X}{|X|}\right) |X|^b \langle X, \nu\rangle^q, \quad X \in M.
    \end{equation}
    In particular, $\gamma$ is unique, and the $k$-convex star-shaped hypersurface $M$ is unique up to homothetic dilations. Furthermore, if $b\leq -k$ and $f$ satisfies \eqref{908998}, then $M$ is strictly convex.
\end{theorem}

\begin{remark}
  When $-b-k=0$ and $q=0$, Guan-Lin-Ma \cite{Guan-Lin-Ma-2006} established the existence, uniqueness, and strict convexity of a $k$-convex solution to equation \eqref{eq:1.4}, along with the existence and uniqueness of the constant $\gamma$.
\end{remark}

This article is organized as follows. Section \ref{Preliminaries} collects the necessary preliminaries.
In Section \ref{Nonhomogeneous-case}, we derive the a priori estimates and establish a constant rank theorem, which guarantees the convexity of the solution and yields the existence and uniqueness result for the case $-b-q-k>0$. In Section \ref{Homogeneous-case}, the corresponding existence, uniqueness, and strict convexity results are proved for the homogeneous case $-b - q - k = 0$.

\section{Preliminaries}\label{Preliminaries}

\subsection{Star-shaped hypersurfaces in $\mathbb{R}^{n+1}$}

We recall some basic formulas for star-shaped hypersurfaces in
\(\mathbb R^{n+1}\); see \cite{Chen-Li-Wang-2018} for details.
Let $(\mathbb{R}^{n+1},\langle \cdot,\cdot\rangle,\bar{\nabla})$ be the $(n+1)$-dimensional Euclidean space equipped with its standard inner product and flat connection, which can be expressed as a warped product manifold $(0,+\infty)\times\mathbb{S}^n$, with the metric
\begin{equation*}
  \langle \cdot,\cdot\rangle=dr^2+r^2\sigma,
\end{equation*}
where $r$ is the distance to a fixed point $o\in\mathbb{R}^{n+1}$ and $\sigma$ is the standard metric of the unit sphere $\mathbb{S}^n$.
The Levi-Civita connection of $\sigma$ on $\mathbb{S}^n$ will be denoted $D$.
The gradient vector field of the function $r$ is denoted by $\partial_r$.

Let $M\subset\mathbb{R}^{n+1}$ be a star-shaped hypersurface with respect to the origin. Then there exists a function $\rho:\mathbb{S}^n\to(0,\infty)$ such that
\[
M=\{X\in\mathbb{R}^{n+1}: X=\rho(x)\,x, \; x\in\mathbb{S}^n\}.
  \]
Take $\{e_1,\dots,e_n\}$ to be a local orthonormal frame field on $\mathbb{S}^n$.
For convenience, we set $\rho_i = D_{e_i} \rho$ and $\rho_{ij} = D^2\rho(e_i, e_j)$. The tangent space of $M$ is spanned at each point by the vectors
\begin{equation*}
X_{i}=\rho e_{i}+\rho_i \partial_r,\quad 1\leq i\leq n.
\end{equation*}
The outward unit normal vector of $M$ is given by
\begin{equation}\label{Nor}
\nu=\frac{1}{v}\left(\partial_r-\frac{1}{\rho^2} \sum_{j=1}^{n}\rho_j e_j\right),
\end{equation}
where $v=\sqrt{1+\rho^{-2}|D \rho|^2}$.
The induced metric $g$ from $\langle \cdot,\cdot\rangle$ has components
\begin{equation}\label{gij}
g_{ij}=\rho^2\delta_{ij}+\rho_i \rho_j.
\end{equation}
Its inverse is
\begin{equation}\label{gij2}
g^{ij}
=
\frac{1}{\rho^2}
\left(
\delta_{ij}
-
\frac{\rho_i\rho_j}{\rho^2v^2}
\right).
\end{equation}
 The second
fundamental form is
\begin{equation}\label{hij}
h_{ij}
=
-\langle \bar\nabla_{X_i}X_j,\nu\rangle
=
\frac{1}{v}
\left(
-\rho_{ij}
+\rho\delta_{ij}
+\frac{2}{\rho}\rho_i\rho_j
\right),
\end{equation}
where \(\bar\nabla\) is the Euclidean connection.
Let $h^i_j=g^{is}h_{sj}$, then
\begin{equation}\label{h_ij}
h^{i}_{j}=\frac{1}{\rho
v}\left(\delta_{ij}+\sum_{k=1}^{n}(-\delta_{ik}+\frac{\rho_i
\rho_k}{\rho^2v^2})(\log \rho)_{jk}\right).
\end{equation}
Let \(\nabla\) denote the Levi-Civita connection of \(g\).
The following identities will be used repeatedly. 
For hypersurfaces in
\(\mathbb R^{n+1}\), the Codazzi equation and the Ricci commutation
formula give
\begin{align}
\nabla_{k}h_{ij}&=\nabla_{j}h_{ik}, \\
\nabla_{i}\nabla_{j}h_{kl}
=\nabla_{k}\nabla_{l}h_{ij}&+h^{m}_{j}
(h_{il}h_{km}-h_{im}h_{kl})+h^{m}_{l}(h_{ij}h_{km}-h_{im}h_{kj}).\label{h-ij-ji}
\end{align}
Fix $X\in M$, and choose a local orthonormal frame $\{E_1,\dots,E_n\}$ around $X$ in $M$ such that $\nabla_{E_i}E_j(X)=0$. At $X$, we have
\begin{align}
(|X|^{2})_{i} &= 2\langle X, E_{i}\rangle, \label{xi} \\
(|X|^{2})_{ij} &= 2\delta_{ij} - 2h_{ij}\langle X, \nu\rangle, \label{xij} \\
\langle X, \nu\rangle_{i} &= h_{ik}\langle X, E_{k}\rangle, \label{ui} \\
\langle X, \nu\rangle_{ij} &= h_{ijk}\langle X, E_{k}\rangle + h_{ij} - h_{ik}h_{kj}\langle X, \nu\rangle. \label{uij}
\end{align}

\subsection{$k$-th elementary symmetric functions}
Let $\lambda=(\lambda_1,\dots,\lambda_n)\in\mathbb{R}^n$, we recall
the definition of elementary symmetric functions for $1\leq k\leq n$,
\begin{equation*}
\sigma_k(\lambda)= \sum _{1 \le i_1 < i_2 <\cdots<i_k\leq
n}\lambda_{i_1}\lambda_{i_2}\cdots\lambda_{i_k}.
\end{equation*}

Let $U$ be a matrix and $\lambda(U)$ be the eigenvalues of $U$.
Throughout this paper, we simply write $\sigma(\lambda(U))$ as $\sigma(U)$, and we denote $\sigma_{k-1}(\lambda|i)=\frac{\partial
\sigma_k}{\partial \lambda_i}$ and
$\sigma_{k-2}(\lambda|ij)=\frac{\partial^2 \sigma_k}{\partial
\lambda_i\partial \lambda_j}$.
In this subsection, we will present some basic properties of the $k$-th elementary symmetric functions, which are useful for the calculations later.

\begin{proposition}\label{sigma}
\cite{Chou-Wang-2001,Hou-Ma-Wu-2010,Hui99} Let $\lambda=(\lambda_1,\dots,\lambda_n)\in\mathbb{R}^n$ and $1\leq
k\leq n$. Then
\begin{enumerate}
\item[(\romannumeral1)]  $\Gamma_1\supset \Gamma_2\supset \cdot\cdot\cdot\supset
\Gamma_n$;

\item[(\romannumeral2)]  $\sigma_{k-1}(\lambda|i)>0$ for $\lambda \in \Gamma_k$ and
$1\leq i\leq n$;

\item[(\romannumeral3)]  $\sigma_k(\lambda)=\sigma_k(\lambda|i)
+\lambda_i\sigma_{k-1}(\lambda|i)$ for $1\leq i\leq n$;

\item[(\romannumeral4)] If $\lambda_1\geq \lambda_2\geq \cdot\cdot\cdot\geq \lambda_n$,
then $$\sigma_{k-1}(\lambda|1)\leq \sigma_{k-1}(\lambda|2)\leq
\cdot\cdot\cdot\leq \sigma_{k-1}(\lambda|n)~~~\text{ for}~~~\lambda \in
\Gamma_k;$$

\item[(\romannumeral5)]
$\sum_{i=1}^{n}\sigma_{k-1}(\lambda|i)=(n-k+1)\sigma_{k-1}(\lambda)$;

\item[(\romannumeral6)]
If $\lambda_1\geq \lambda_2\geq \cdot\cdot\cdot\geq \lambda_n$,
then
$\lambda_1 \sigma_{k-1}(\lambda |1) \geq \frac{k}{n}\sigma_k(\lambda)$;

\item[(\romannumeral7)]
$\sum_{i=1}^{n}\lambda_{i}^2\sigma_k(\lambda | i)=\sigma_1(\lambda)\sigma_{k+1}(\lambda)-(k+2)\sigma_{k+2}(\lambda)$.
\end{enumerate}
\end{proposition}

The generalized Newton-MacLaurin inequality is as follows.
\begin{proposition}\label{NM}
\cite{Spruck-2005} If $\lambda \in \Gamma_m$, $m > l \geq 0$, $ r > s \geq 0$, $m
\geq r$ and $l \geq s$, then
\begin{align}
\left(\frac{{\sigma _m (\lambda )}/{C_n^m }}{{\sigma _l (\lambda
)}/{C_n^l }}\right)^{\frac{1}{m-l}} \le \left(\frac{{\sigma _r
(\lambda )}/{C_n^r }}{{\sigma _s (\lambda )}/{C_n^s
}}\right)^{\frac{1}{r-s}}. \notag
\end{align}
\end{proposition}



\section{Existence and uniqueness in nonhomogeneous case}\label{Nonhomogeneous-case}

In this section, we first derive a priori estimates and then use them to obtain the existence and uniqueness of a $k$-convex hypersurface for equation \eqref{G-eqk} in the case $-b-q-k>0$. Next, under Assumption \eqref{908998}, the constant rank theorem (Theorem \ref{crt}) is applied to prove the strict convexity of the obtained $k$-convex hypersurface.

Let $M$ be a star-shaped hypersurface satisfying equation \eqref{G-eqk}. Then equation \eqref{G-eqk} can be rewritten as
\begin{equation}\label{eq-bqk-geq-0}
\sigma_k(\kappa)=f \rho^{b+2q} (\rho^2 + |D \rho|^2)^{-\frac{q}{2}}, \quad \text{on } \quad \mathbb{S}^n.
\end{equation}
We now establish the a priori estimates for equation \eqref{eq-bqk-geq-0}.

\subsection{The a priori estimates}

\begin{theorem}\label{c0}
Let $1\leq k <n $ and $-b-q-k>0$. Assume that $f \in C^2(\mathbb{S}^n)$ is a positive function
and that $M$ is a $k$-convex star-shaped hypersurface satisfying  equation \eqref{eq-bqk-geq-0}.
Then
$$
\left(\frac{ \min_{ \mathbb{S}^n} f}{C_n^k }\right)^{\frac{1}{-b-q-k}}
\leq \rho(x)
\leq \left(\frac{ \max_{ \mathbb{S}^n} f}{C_n^k }\right)^{\frac{1}{-b-q-k}},       \quad \forall~x \in \mathbb{S}^n.
$$
\end{theorem}
\begin{proof}
Assume $\rho$ attains its maximum value at $x_0 \in \mathbb{S}^n$. Then $D \rho(x_0) = 0$ and $D^2 (\log \rho)(x_0) \leq 0$.
By \eqref{h_ij}, the principal curvatures $\kappa=(\kappa_1,\dots,\kappa_n)$ of $M$ satisfy
$$\kappa_i(x_0) \geq (\max\limits_{\mathbb{S}^n}\rho)^{-1}, \quad  \forall~~1 \leq i \leq n.$$
Therefore,
\begin{equation*}
C_n^k \left( \max\limits_{\mathbb{S}^n}\rho \right)^{-k}
\le \sigma_k(\kappa)
= f\,\rho^{b+2q}\bigl(\rho^2+|D \rho|^2\bigr)^{-q/2}
\le \bigl(\max_{\mathbb S^n}f\bigr)\bigl(\max_{\mathbb S^n}\rho\bigr)^{b+q},
\end{equation*}
which implies that
$$(\max_{\mathbb{S}^n} \rho)^{-b-q-k} \leq (C_n^k)^{-1} \max_{ \mathbb{S}^n} f.$$
A similar argument yields
 $$(C_n^k)^{-1}\min_{ \mathbb{S}^n} f \leq (\min_{\mathbb{S}^n} \rho)^{-b-q-k}.$$
\end{proof}

Next, we establish the gradient estimates. Let $B=(B^{ij})$ be the square root of the positive definite matrix $(g^{ij})$. To simplify the calculation, we denote
\begin{equation*}
  u=-\log \rho,\quad
  \bar{g}^{i j} =\delta_{i j}-\frac{u_{i} u_{j}}{\sqrt{1+|D u|^{2}}\big(1+\sqrt{1+|D u|^{2}}\big)} ,\quad
  \bar{h}_{l m} =\delta_{l m}+u_{l} u_{m}+u_{l m}.
\end{equation*}
Then $B^{ij}=e^u \bar{g}^{ij}$ and $h_{ij}=e^{-u}(1+|D u|^2)^{-\frac{1}{2}}\bar{h}_{ij}$. Since $\kappa=(\kappa_1,\dots,\kappa_n) $ are the eigenvalues of the matrix $ (B^{il}h_{lm}B^{mj}) $, we obtain
\begin{equation}\label{kappa-B-h}
  \sigma_k(\kappa)=\sigma_k(B^{il}h_{lm}B^{mj})=e^{ku}(1+|D u|^2)^{-\frac{k}{2}}\sigma_k(a_{ij}),
\end{equation}
where $(a_{ij})=(\bar{g}^{il}\bar{h}_{lm}\bar{g}^{mj})$. Let $F(a_{ij})=\sigma_{k}(a_{ij})$.
It follows from \eqref{kappa-B-h} that equation \eqref{eq-bqk-geq-0} is equivalent to
\begin{equation}\label{c121}
F(a_{ij})=fe^{(-b-q-k)u}\left(1+|D u|^{2}\right)^{\frac{k-q}{2}}, \quad\text{on}\quad\mathbb{S}^n.
\end{equation}
With these preparations, we establish the following gradient estimates:

\begin{theorem}\label{c122}
Let $1\leq k <n $ and $-b-q-k>0$. Suppose that $f \in C^2(\mathbb{S}^n)$ is a positive function
and that $M$ is a $k$-convex star-shaped hypersurface satisfying equation \eqref{c121}.
Then
$$\max_{\mathbb{S}^{n}}| D \log\rho|\leq C,$$
where $C$ is a constant depending on $n, b, q, k$ and $\max_{\mathbb{S}^{n}} \frac{| D f |}{f}$. In particular,
$$1\leq\frac{\max_{\mathbb{S}^{n}}\rho}{\min_{\mathbb{S}^{n}}\rho}\leq e^C.$$
\end{theorem}

\begin{proof}

Define the function $Q=|D u|^{2}$, and assume that $Q$ attains its maximum at $x_{0}\in\mathbb{S}^n$. Choose a local coordinate frame field $\{e_1,\dots,e_n\}$ around $x_{0}$ such that $D_{e_i}e_j(x_0)=0$, $u_{1}(x_0)=|D u|$, and $(u_{rs}(x_0))_{2\leq r,s\leq n}$ is diagonal.
In the following, all computations are performed at $  x_0  $.  Since \(Q\) attains
its maximum at \(x_0\), one has
$$0=Q_{i}=2\sum_{m=1}^n u_{m}u_{mi}=2u_{1}u_{1i},$$
it implies that
$ u_{1i}=0~~~\text{for all}~~~i=1,\dots,n.$
Therefore, $(u_{ij})$, $(\bar{g}^{ij})$, $(\bar{h}_{lm})$, and $(a_{ij})$ are all diagonal at $x_{0}$. Direct computation shows that
\begin{equation*}
  \bar{g}^{11}=(1+|D u|^{2})^{-\frac{1}{2}},\quad\bar{h}_{11}=1+|D u|^{2},\quad a_{11}=1,
\end{equation*}
and
\begin{equation*}
  \bar{g}^{ii}=1,\quad\bar{h}_{ii}=a_{ii}=1+u_{ii},\quad\text{for}\quad 2\leq i\leq n.
\end{equation*}
Without loss of generality, assume  that $u_{22}\geq\cdots\geq u_{nn}$.
Let ${F}^{ij}=\frac{\partial F}{\partial a_{ij}}$. Then $(F^{ij})$ is diagonal at $x_{0}$.
By equation \eqref{c121},
\begin{equation}\label{c111a}
\begin{aligned}
{F}^{ij}a_{ijs}=e^{(-b-q-k)u}(1+|D u|^{2})^{\frac{k-q}{2}}\left[(-b-q-k)fu_{s}+f_{s}\right].
\end{aligned}
\end{equation}
Since $u_s(\bar{g}^{il})_s=u_s (\bar{g}^{mj})_s=0$, we obtain
$$u_{s}a_{ijs}=u_{s}(\bar{g}^{il}\bar{h}_{lm}\bar{g}^{mj})_{s}
=\bar{g}^{il}u_{s}\bar{h}_{lms}\bar{g}^{mj}
=\bar{g}^{il}u_{s}u_{lms}\bar{g}^{mj},$$
which implies that
\begin{equation}\label{c111b}
\begin{aligned}
\sum^n_{s=1}u_{s}{F}^{ij}a_{ijs}
=&\sum^n_{s=1}{F}^{ij}\bar{g}^{il}u_{s}
\left(u_{slm}-u_{s}\delta_{lm}+u_{m}\delta_{ls}\right)\bar{g}^{mj}\\
=&\sum^n_{s=1}{F}^{ij}\bar{g}^{il}\bar{g}^{mj}u_{s}u_{slm}-|D u|^{2}{F}^{ij}\bar{g}^{il}\bar{g}^{lj}+{F}^{ij}\bar{g}^{il}\bar{g}^{mj}u_{m}u_{l}.
\end{aligned}
\end{equation}
Set $\widetilde{G}^{ij}={F}^{lm}\overline{g}^{il}\overline{g}^{mj}$. Combining \eqref{c111a} and \eqref{c111b} yields
$$\sum^n_{s=1}\widetilde{G}^{lm}u_{s}u_{slm}=|D u|^{2}\sum_{i=2}^{n}\widetilde{G}^{ii}+e^{(-b-q-k)u}(1+|D u|^{2})^{\frac{k-q}{2}}\left[(-b-q-k)f|D u|^{2}+|D u|f_{1}\right].$$
A straightforward computation then gives
\begin{align}
 0\geq \frac{1}{2}\widetilde{G}^{ij}Q_{ij}=&\sum^n_{m=1}\widetilde{G}^{ij}u_{mi}u_{mj}+\sum^n_{m=1}\widetilde{G}^{ij}u_{m}u_{mij} \notag\\
 \geq&\sum_{i=2}^{n}\widetilde{G}^{ii}u_{ii}^{2}+|D u|^{2}\sum_{i=2}^{n}\widetilde{G}^{ii}+e^{(-b-q-k)u}(1+|D u|^{2})^{\frac{k-q}{2}}f_{1}|D u|\notag\\
 &+(-b-q-k)e^{(-b-q-k)u}(1+|D u|^{2})^{\frac{k-q}{2}}f|D u|^{2}\notag\\
 \geq &\sum_{i=2}^{n}{F}^{ii}u_{ii}^{2}+|D u|^{2}\sum_{i=2}^{n}{F}^{ii}- F\frac{|D f|}{f}|D u|+(-b-q-k)F|D u|^{2}\notag\\
=&\sum_{i=1}^n F^{ii} a_{ii}^2 -2kF+ (1+|D u|^{2}) \sum_{i=2}^n F^{ii}+F^{11}- F \frac{|D f|}{f}|D u|\notag\\
&+(-b-q-k)F|D u|^{2}\notag\\
\geq&\sum_{i=1}^n F^{ii} a_{ii}^2 -2kF + (1+|D u|^{2}) \sum_{i=2}^n F^{ii}-\frac{|D f|^2F}{4(-b-q-k)f^2}
\label{equiiur-01},
\end{align}
where we used $-b-q-k>0$ and $\widetilde{G}^{ii}={F}^{ii}$ for $2\leq i\leq n$.
Let $\xi=\left(a_{11}(x_0),\dots,a_{nn}(x_0)\right)$ and $\xi_i=a_{ii}(x_0)$. Then $\xi\in\Gamma_k$.
Noting that $  \xi_1=1  $, by (iii) and (v) of Proposition \ref{sigma} together with Proposition \ref{NM}, we obtain
\begin{equation}\label{equiiur-011-12}
\begin{aligned}
\sum_{i=2}^n F^{ii}=\sum_{i=2}^{n}\sigma_{k-1}(\xi |i)=(n-k)\sigma_{k-1}(\xi)+\sigma_{k-2}(\xi|1)\geq \frac{k(n-k)}{n-k+1}(C_n^k)^{\frac{1}{k}} F^{1-\frac{1}{k}}.
\end{aligned}
\end{equation}
By (vii) of Proposition \ref{sigma} and Proposition \ref{NM},
\begin{equation}\label{equiiur-011-13}
  \sum_{i=1}^n F^{ii} \xi_i^2= \sigma_1(\xi)\sigma_k(\xi)-(k+1)\sigma_{k+1}(\xi)
  \geq n(C_n^k)^{-\frac{1}{k}}F^{1+\frac{1}{k}}-(k+1)\sigma_{k+1}(\xi).
\end{equation}
It follows from (3.33) in \cite{Guan-Lin-Ma-2006} and the assumption $k<n$ that
\begin{equation}
\label{eq:claim}
(k+1)\sigma_{k+1}(\xi) \leq (k+1)F + (n-k-1)(C_{n-1}^k)^{-\frac{1}{k}} F^{1+\frac{1}{k}}.
\end{equation}
 Combining \eqref{equiiur-011-13} and \eqref{eq:claim} yields
\begin{equation}\label{equiiur-011-13333}
\begin{aligned}[t]
 \sum_{i} F^{ii} \xi_i^2
\geq
c_{0}
(C_{n}^k)^{-\frac{1}{k}} F^{1+\frac{1}{k}}-(k+1)F,
\end{aligned}
\end{equation}
where $c_{0}=n-(n-k-1)n^{\frac{1}{k}}(n-k)^{-\frac{1}{k}}>0$.
From \eqref{equiiur-01}, \eqref{equiiur-011-12}, and \eqref{equiiur-011-13333},
it follows that
\begin{align}
 0\geq&~c_0(C_{n}^k)^{-\frac{1}{k}} F^{\frac{1}{k}}
 +\frac{k(n-k)}{n-k+1}(C_n^k)^{\frac{1}{k}} F^{-\frac{1}{k}}(1+|D u|^{2})\label{c1222221}
 -\frac{|D f|^2}{4(-b-q-k)f^2}-(3k+1)\notag\\
 \geq&~2c_0^{\frac{1}{2}}k^{\frac{1}{2}}
 (n-k)^{\frac{1}{2}}(n-k+1)^{-\frac{1}{2}}(1+|D u|^{2})^{\frac{1}{2}}
 -\frac{|D f|^2}{4(-b-q-k)f^2}-(3k+1)\notag,
\end{align}
which implies that $|D u|\leq C$,
where $C$ depends on $n$, $b$, $q$, $k$, and $\max_{\mathbb{S}^{n}} \frac{|D f|}{f}$.

Set $\widetilde{\rho}=\frac{\rho}{\min_{\mathbb{S}^{n}} \rho}$. Since $|D \log \widetilde{\rho}|=|D\log \rho| \leq C$ and $\min_{\mathbb{S}^{n}} \log \widetilde{\rho}=0$, we obtain
\begin{equation*}
  1\leq\max_{\mathbb{S}^{n}} \widetilde{\rho}=\frac{\max_{\mathbb{S}^{n}}\rho}{\min_{\mathbb{S}^{n}} \rho} \leq e^C,
\end{equation*}
which completes the proof.
\end{proof}

Applying Theorem 3.5 in \cite{GongTu20262}, we derive the following curvature estimates:
\begin{theorem}\label{c22}
Let $1\leq k <n $ and  $q\leq 1$. Suppose $f \in C^2(\mathbb{S}^n)$ is a positive function
and $M$ is a $k$-convex star-shaped hypersurface satisfying equation \eqref{G-eqk}. Then
\begin{equation}\label{eq:3.1}
\sigma_1(\kappa) \leq C.
\end{equation}
 where the constant $C$ depends on $n$, $k$, $b$, $q$, $|\rho|_{C^1}$, $|f|_{C^2}$, $\min_{\mathbb{S}^n}f$, and $\min_{\mathbb{S}^n}\rho$.
\end{theorem}

\subsection{A constant rank theorem}

When applying the continuity method, Theorem \ref{crt} ensures that the convexity of solutions is preserved. The proof follows the method used in the proofs of Theorems 1.1 and 1.2 in \cite{CaffGpMa}. Before giving the proof, some notation is introduced.

Let $W=(W_{ij})$ with $W_{ij}=h_{ij}$, and define $\tilde{F}(W)=-\sigma_k^{-\frac{1}{k}}(W)$.
Denote
\begin{equation*}
  \tilde{F}^{ij}=\frac{\partial\tilde{F}}{\partial W_{ij}},\quad \tilde{F}^{ij,rs}=\frac{\partial^2 \tilde{F}}{\partial W_{ij}\partial W_{rs}},\quad \psi(X)=f^{-\frac{1}{k}}\left(\frac{X}{|X|}\right)|X|^{-\frac{b}{k}}.
\end{equation*}
For $X\in M$, set $\tilde{\psi}(X)=-\psi(X)\langle X, \nu(X) \rangle^{-\frac{q}{k}}$.
Then equation \eqref{G-eqk} is equivalent to
\begin{equation}\label{full-rank-equation}
 \tilde{F}(W)=\tilde{\psi}(X),\quad X \in M.
\end{equation}

We now prove the following constant rank theorem.

\begin{theorem}\label{crt}
Let $1 \leq k < n $.
Suppose that $M$ is a star-shaped $C^3$-hypersurface satisfying equation \eqref{G-eqk}, that the second fundamental form of $M$ is positive semi-definite, and that $f\in C^2(\mathbb{S}^n)$ is a positive function satisfying \eqref{908998}. Then $M$ is strictly convex.
\end{theorem}

\begin{proof}

Let $O\subset M$ be an open neighborhood of some point $p_0\in M$ where the minimum rank $l$ of $W$ is attained. For any $X_0\in O$, let $\kappa_n\geq\dots\geq\kappa_1$ be the eigenvalues of $W$ at $X_0$. There is a positive constant $C$ depending only on $n$, $|\rho|_{C^3}$, and $|\tilde{\psi}|_{C^2}$ such that $\kappa_n\geq\dots\geq\kappa_{n-l+1}\geq C$. Let $G=\{ n-l+1,\dots,n  \}$ and $B=\{1,\dots, n-l \}$ be the ``good" and ``bad" sets of indices, respectively. Let $\Lambda_G=(\kappa_{n-l+1},\dots,\kappa_n )$ be the good eigenvalues of $W$ at $X_0$, for simplicity of notation, we also write $G=\Lambda_G$ if there is no confusion.

Since \(O\) is sufficiently small and \(\widetilde F\) is elliptic, a
constant \(A>0\) can be chosen so large that
\begin{equation*}
   \min_\alpha \widetilde{F}^{\alpha\alpha}
   \geq
   \frac{100}{A}
   \sum_{\alpha,\beta,s,\xi}
   \left|
   \widetilde{F}^{\alpha\beta,s\xi}(W(X))
   \right|,
   \quad X\in O.
\end{equation*}
Consider the test function
\begin{equation*}
   \Phi(X)=\sigma_{l+1}(W)+A\sigma_{l+2}(W).
\end{equation*}
Following the notations in \cite{CaffGpMa}, we say that $h(y)\lesssim k(y)$ provided there exists positive constants $C_1$ and $C_2$ such that
\[
(h-k)(y)
\leq
\left(C_1|\nabla\Phi|+C_2\Phi\right)(y).
\]
Also  write $h(y)\sim k(y)$ if $h(y)\lesssim k(y)$ and $k(y)\lesssim h(y)$.
In the following, all computations are performed at $X_0$.

We want to show that
\begin{equation}\label{Ineq-full-rank-I}
    \frac{1}{\sigma_l(G)}\sum_{\alpha=1}^{n}
    \tilde{F}^{\alpha\alpha}\Phi_{\alpha\alpha}\lesssim 0.
\end{equation}
To prove \eqref{Ineq-full-rank-I}, we may assume $\rho\in C^4$ by approximation. For each $X_0\in O$ fixed,  choose a local orthonormal frame $\{E_1,\dots,E_n\}$ in a neighborhood of $X_0$ in $M$ such that $W$ is diagonal at $X_0$ and $\nabla_{E_i}E_j(X_0)=0$.
Similar to (2.19) in \cite{CaffGpMa}, one has
\begin{equation}\label{rank-F-Phi}
    \begin{split}
        \sum_{\alpha=1}^{n}\tilde{F}^{\alpha\alpha}\Phi_{\alpha\alpha}&\sim
        A\sum_{i,\alpha=1}^{n}\tilde{F}^{\alpha\alpha}
        \sigma_{l+1}(W|i)W_{ii\alpha\alpha}+
        \sigma_l(G)\sum_{\alpha=1}^{n}\sum_{i\in B}\tilde{F}^{\alpha\alpha}W_{ii\alpha\alpha}\\
        &-\sigma_{l-1}(G)\sum_{\alpha=1}^{n}\sum_{i,j\in B}\tilde{F}^{\alpha\alpha}W_{ij\alpha}^2
        -2\sum_{\alpha=1}^{n}\sum_{i\in B,j\in G}\sigma_{l-1}(G|j)\tilde{F}^{\alpha\alpha}W_{ij\alpha}^2 \\
        &-A\sigma_l(G)\sum_{\alpha=1}^{n}\sum_{i,j\in B}\tilde{F}^{\alpha\alpha}W_{ij\alpha}^2.
    \end{split}
\end{equation}
By \eqref{h-ij-ji},
\begin{equation}\label{W-commutation-formula-rank}
    \begin{split}
        W_{ii\alpha\alpha}&=h_{ii\alpha\alpha}
        =h_{\alpha\alpha ii}+h_{\alpha\alpha}h_{ii}^2-h_{ii}h_{\alpha\alpha}^2
        =W_{\alpha\alpha ii}+W_{\alpha\alpha}W_{ii}^2-W_{ii}W_{\alpha\alpha}^2.
    \end{split}
\end{equation}
Differentiating \eqref{full-rank-equation} twice gives
\begin{equation}\label{rank-Second-order-derivative}
  \tilde{F}^{\alpha\alpha}W_{\alpha\alpha ii}=\tilde{\psi}_{ii}-
  \tilde{F}^{\alpha\beta,rs}W_{\alpha\beta i}W_{rsi}.
\end{equation}
Note that $W_{ii}\sim 0$ for $i\in B$ and that $\sigma_{l+1}(W|i)\sim 0$ for $1\leq i\leq n$. Therefore, by \eqref{rank-F-Phi}-\eqref{rank-Second-order-derivative},
\begin{equation}\label{rank-F-I}
  \begin{split}
    &\quad \sum_{\alpha=1}^{n}\tilde{F}^{\alpha\alpha}\Phi_{\alpha\alpha} \\
    &
    \sim \sigma_l(G)\left [ \sum_{i\in B}\tilde{\psi}_{ii}-\sum_{i\in B}\sum_{\alpha,\beta,r,s=1}^{n}\tilde{F}^{\alpha\beta, rs}W_{\alpha\beta i}W_{rsi}-A\sum_{\alpha=1}^{n}\sum_{i,j\in B}\tilde{F}^{\alpha\alpha}W_{ij\alpha}^2 \right] \\
    &\quad -\sigma_{l-1}(G)\sum_{\alpha=1}^{n}\sum_{i,j\in B}\tilde{F}^{\alpha\alpha}W_{ij\alpha}^2
    -2\sum_{\alpha=1}^{n}\sum_{i\in B,j\in G}\sigma_{l-1}(G|j)\tilde{F}^{\alpha\alpha}W_{ij\alpha}^2.
  \end{split}
\end{equation}
Following the same proof as that of (2.21)-(2.27) in \cite{CaffGpMa}, and the inverse convexity  of
\(-\sigma_k^{-1/k}\), we obtain
\begin{equation}\label{rank-F-II}
  \frac{1}{\sigma_l(G)}\sum_{\alpha=1}^n
    \tilde{F}^{\alpha\alpha}\Phi_{\alpha\alpha}
    \lesssim \sum_{i\in B}\tilde{\psi}_{ii}.
\end{equation}

It remains to estimate the right-hand side. Since \(h_{ii}\sim0\) for
\(i\in B\), \eqref{ui} gives
 $\langle X, \nu \rangle_i=h_{ii}\langle X,E_i\rangle \sim 0$ for $i\in B$, and hence
\begin{equation}\label{X-nu-q-k-i}
  \left(\langle X,\nu\rangle^{-\frac{q}{k}}\right)_i\sim 0\quad\text{for}\quad i\in B.
\end{equation}
By \eqref{uij} and $\sum_{i\in B}W_{ii\alpha}\sim 0$ for $1\leq\alpha\leq n$,
\begin{equation*}
  \sum_{i\in B} \langle X,\nu\rangle_{ii}=
  \sum_{i\in B}( h_{ii}+\sum_{j=1}^{n}h_{iij}\langle X, E_j\rangle-h_{ii}^2\langle X,\nu\rangle )\sim 0.
\end{equation*}
Therefore,
 \begin{equation}\label{X-nu-nabla}
   \begin{split}
   \sum_{i\in B}\bar{\nabla}^2\left(\langle X,\nu\rangle^{-\frac{q}{k}}\right)(E_i,E_i)
   =\sum_{i\in B} \left[\left(\langle X,\nu\rangle^{-\frac{q}{k}}\right)_{ii}+h_{ii}\partial_{\nu}\left(\langle X,\nu\rangle^{-\frac{q}{k}}\right)\right] \sim 0,
   \end{split}
 \end{equation}
where $ \bar{\nabla}^2 $ denotes the Hessian operator on $ \mathbb{R}^{n+1} $ and $ \partial_{\nu} $ denotes the directional derivative along $ \nu $.
Using \(h_{ii}\sim0\) for \(i\in B\), \eqref{908998},
\eqref{X-nu-q-k-i}, and \eqref{X-nu-nabla}, one obtains
\begin{equation}\label{psi-ii-B-leq-0}
  \begin{split}
  \sum_{i\in B}\tilde{\psi}_{ii} = &
  \sum_{i\in B} \bar{\nabla}^2 \tilde{\psi}(E_i,E_i) -\sum_{i\in B} h_{ii}\partial_{\nu}\tilde{\psi}\\
  = & -\langle X,\nu\rangle^{-\frac{q}{k}}\sum_{i\in B}\bar{\nabla}^2\psi(E_i,E_i) -\psi \sum_{i\in B}\bar{\nabla}^2\left(\langle X,\nu\rangle^{-\frac{q}{k}}\right)(E_i,E_i) \\
  &-2\sum_{i\in B}\psi_i \left(\langle X,\nu\rangle^{-\frac{q}{k}}\right)_i
  -\sum_{i\in B} h_{ii}\partial_{\nu}\tilde{\psi}\\
  \sim & -\!\langle X,\nu\rangle^{-\frac{q}{k}}\sum_{i\in B}\bar{\nabla}^2\psi(E_i,E_i) \lesssim 0,
  \end{split}
\end{equation}
where we used $\langle X,\nu\rangle^{-\frac{q}{k}}>0$.
The combination of \eqref{rank-F-II} and \eqref{psi-ii-B-leq-0} leads to \eqref{Ineq-full-rank-I}. The strong minimum principle implies that $W$ is of constant rank. Finally, assume $\rho(x)$ attains its maximum $\tilde{\rho}$ at $\tilde{x}\in \mathbb{S}^n$. Let $\tilde{X}=\tilde{\rho}\tilde{x}\in M$. It follows from \eqref{h_ij} that $W(\tilde{X})\geq \tilde{\rho}^{-1} I$, which
 implies that $W$ is of full rank at $\tilde{X}$. Hence $W$ is of full rank.
\end{proof}

\subsection{Existence, uniqueness, and strict convexity}

Let $\mathcal{F}(x,\rho,D\rho,D^2\rho)=\sigma^{\frac{1}{k}}(W)$ and $\mathcal{K}(x,\rho,D\rho)=f^{\frac{1}{k}}(x) \rho^{\frac{b+2q}{k}} (\rho^2 + |D \rho|^2)^{-\frac{q}{2k}}$, where $x\in\mathbb{S}^n$. Equation \eqref{G-eqk} can be written as
\begin{equation}\label{eq-uniqueness-existence-convexity}
  \mathcal{F}(x,\rho,D\rho,D^2\rho)=\mathcal{K}(x,\rho,D\rho).
\end{equation}

The following two lemmas will be used to prove the existence and uniqueness of solutions. Since their proofs are similar to those of Lemmas 2.4 and 2.5 in \cite{Guan09}, they are omitted.

\begin{lemma}\label{uni}
Let $1\leq k<n$ and $-b-q-k>0$. Suppose $M_i=\{X\in\mathbb{R}^{n+1}:X=\rho_i(x)x, ~ x \in \mathbb{S}^n\}$ for $i=1,2$ are two $k$-convex star-shaped hypersurfaces satisfying equation \eqref{G-eqk}. Then $\rho_1 \equiv \rho_2$ on~ $\mathbb{S}^n$.
\end{lemma}

\begin{lemma}\label{linearized}
Let $1\leq k<n$ and $-b-q-k>0$. Let $\mathcal{L}$ denote the linearized operator of $\mathcal{F}(x,\rho,D\rho,D^2\rho)-\mathcal{K}(x,\rho,D\rho)$ at a solution $\rho$ of equation \eqref{G-eqk}.
If $w$ satisfies $\mathcal{L}w = 0$ on $\mathbb{S}^n$, then $w \equiv 0$ on $\mathbb{S}^n$.
\end{lemma}

\begin{proof}[\textbf{Proof of Theorem~\ref{Th-One}}]
For $1\leq k\leq n-1$, we consider the following family of equations
\begin{equation}\label{ft}
F_t=\sigma_k(\kappa)=f_t \rho^{b+2q} (\rho^2 + |D \rho|^2)^{-q/2}, \quad \text{on }\quad \mathbb{S}^n,
\end{equation}
where $f_t=(1-t+t f^{-\frac{1}{k}})^{-k}$ with $0 \leq t \leq 1$. Let
$T = \{ t \in [0,1] : \text{equation \eqref{ft} has a }$ $ k\text{-convex solution} \}$.
$T$ is nonempty because $\rho = (C_n^k)^{\frac{1}{b+q+k}}$ is a solution for $t = 0$. By Lemma \ref{linearized}, the linearized operator is invertible at any solution; hence by the implicit function theorem, $T$ is open. It follows from Theorem \ref{c0}, Theorem \ref{c122}, Theorem \ref{c22}, and the Evans-Krylov theorem that $|\rho_t|_{C^{3,\alpha}(\mathbb{S}^n)} \leq C$,
where $C$ depends on $n$, $k$, $b$, $q$, $\alpha$, $\min_{\mathbb{S}^n} f$, $\max_{\mathbb{S}^n} f$, $|f|_{C^1}$, and $|f|_{C^2}$.
Therefore, $T$ is closed.
Thus $T = [0,1]$, and a $k$-convex solution of \eqref{ft} exists for every $ t\in[0,1]$.
The uniqueness part of the theorem follows from Lemma \ref{uni}.

Next, we prove the strict convexity of the solution under \eqref{908998} and the assumption $b\leq -k$. For $t\in [0,1]$, let $\rho_t$ be the unique $k$-convex solution of \eqref{ft} and $M_t=\{X\in\mathbb{R}^{n+1}:X=\rho_t(x)x, ~ x \in \mathbb{S}^n\}$. Denote $\tilde{T}=\{t\in [0,1]: M_t \text{ is strictly convex}\}$. It follows from $0\in\tilde{T}$ that $\tilde{T}$ is nonempty. Let $t^{*}$ be an arbitrary limit point of $\tilde{T}$. Then the second fundamental form of $M_{t^{*}}$ is positive semi-definite. It follows from \eqref{908998} and $b\leq -k$ that $f_{t^{*}}^{-\frac{1}{k}}(\frac{X}{|X|}) |X|^{-\frac{b}{k}}$ is locally convex in $\mathbb{R}^{n+1}  \setminus \{0\}$. By Theorem \ref{crt}, $M_{t^{*}}$ is strictly convex, which implies $t^{*}\in \tilde{T}$. Therefore, $\tilde{T}$ is closed. Note that $\tilde{T}$ is also open; hence $\tilde{T}=[0,1]$. This completes the proof.
\end{proof}

\section{Existence and uniqueness in homogeneous case}\label{Homogeneous-case}

In this section, we consider equation \eqref{G-eqk} for the case $-b-q -k=0$.  We study the following equation
\begin{equation}
\label{eq:5.1}
\sigma_k(\kappa)=f \rho^{b+2q-\varepsilon} (\rho^2 + |D \rho|^2)^{-\frac{q}{2}}, \quad \text{on } \quad \mathbb{S}^n,
\end{equation}
where $\varepsilon > 0$ and $\kappa=(\kappa_1,\dots,\kappa_n)$ are the principal curvatures of $M=\{X\in\mathbb{R}^{n+1}:X=\rho(x)x, ~ x \in \mathbb{S}^n\}$.

\subsection{The a priori estimates}

Following the proof of Theorem \ref{c0}, we derive $C^0$ estimates for equation \eqref{eq:5.1}.

\begin{theorem}
\label{thm:5.1}
Let $1\leq k <n $, $\varepsilon > 0$, and $-b-q-k=0$. Assume that $f \in C^2(\mathbb{S}^n)$ is a positive function
and that $M$ is a $k$-convex star-shaped hypersurface satisfying  equation \eqref{eq:5.1}.
Then
$$
\frac{ \min_{\mathbb{S}^n} f}{C_n^k} \leq \rho^{\varepsilon}(x) \leq \frac{ \max_{ \mathbb{S}^n} f}{C_n^k}, \quad \forall~x \in \mathbb{S}^n.
$$
\end{theorem}

Based on $C^0$ estimates, we can obtain the following $C^1$ estimates.
\begin{theorem}\label{c1225}
Let $1\leq k <n $, $\varepsilon > 0$, $q \neq 0$, and $-b-q-k=0$. Suppose that $f \in C^2(\mathbb{S}^n)$ is a positive function
and that $M$ is a $k$-convex star-shaped hypersurface satisfying equation \eqref{eq:5.1}.
Then
\begin{equation}\label{9i898u9uer55}
\max_{\mathbb{S}^{n}}|D\log\rho|\leq C,
\end{equation}
where $C$ is a constant depending on $n$, $k$, $q$, $\min_{\mathbb{S}^{n}}f$, $\max_{\mathbb{S}^{n}}f$, and $\max_{\mathbb{S}^{n}}\frac{|D f|}{f}$. In particular,
\begin{equation}\label{9i898u9uer155}
1\leq\frac{\max_{\mathbb{S}^{n}}\rho}{\min_{\mathbb{S}^{n}}\rho}\leq e^C.
\end{equation}
\end{theorem}
\begin{proof}
Recall $u=-\log \rho$ and $F(a_{ij})=\sigma_{k}(a_{ij})$. Similar to the derivation in \eqref{c121}, equation \eqref{eq:5.1} is equivalent to
\begin{equation}\label{eq:5.4}
F(a_{ij})=fe^{\varepsilon u}\left(1+|D u|^{2}\right)^{\frac{k-q}{2}},
\quad\text{on}\quad\mathbb{S}^n.
\end{equation}
Define the function $Q=|D u|^{2}$, and assume that $Q$ attains its maximum at $x_{0}\in\mathbb{S}^n$. Choose a local coordinate frame field $\{e_1,\dots,e_n\}$ around $x_{0}$ such that $D_{e_i}e_j(x_0)=0$, $u_{1}(x_0)=|D u|$, and $(u_{rs}(x_0))_{2\leq r,s\leq n}$ is diagonal.
In the following, all computations are performed at $  x_0  $.
By an argument similar to that used in the proof of Theorem \ref{c122},
\begin{equation}\label{equiiur-0117675}
\begin{aligned}[t]
0 \geq&\sum_{i} F^{ii} a_{ii}^2 -2kF+ (1+|D u|^{2}) \sum_{i>1} F^{ii}+F^{11}
- F|D u| \frac{|D f|}{f}+\varepsilon F|D u|^{2}.\\
\end{aligned}
\end{equation}
By \eqref{equiiur-011-12}, \eqref{equiiur-011-13333}, and \eqref{equiiur-0117675},
\begin{equation}\label{c12222215}
 0\geq c_0 (C_{n}^k)^{-\frac{1}{k}} F^{\frac{1}{k}}- |D u| \frac{|D f|}{f}
 +\frac{k(n-k)}{n-k+1}(C_n^k)^{\frac{1}{k}} F^{-\frac{1}{k}}(1+|D u|^{2})-(3k+1).
\end{equation}
It follows from $q \neq 0$ and Theorem \ref{thm:5.1} that
\begin{eqnarray*}
\begin{aligned}
C_n^k\cdot \frac{ \min_{\mathbb{S}^{n}} f}{\max_{\mathbb{S}^{n}} f} \leq e^{\varepsilon u} f \leq  C_n^k\cdot \frac{\max_{\mathbb{S}^{n}} f}{\min_{\mathbb{S}^{n}} f}.
\end{aligned}
\end{eqnarray*}
Therefore,
\begin{eqnarray*}
\begin{aligned}
 0\geq~C(n, k)  \left(\frac{ \min_{\mathbb{S}^{n}} f}{\max_{\mathbb{S}^{n}} f}\right)^{\frac{1}{k}}\left((1+|D u|^{2})^{\frac{1}{2}-\frac{q}{2k}}+(1+|D u|^{2})^{\frac{1}{2}+\frac{q}{2k}}\right)-(3k+1)- |D u| \frac{|D f|}{f},
\end{aligned}
\end{eqnarray*}
which implies $|D u|\leq C$,
  where $C$ is a constant depending on $n$, $k$, $q$, $\min_{\mathbb{S}^{n}}f$, $\max_{\mathbb{S}^{n}}f$, and $\max_{\mathbb{S}^{n}} \frac{|D f|}{f}$.
As in the proof of Theorem \ref{c122},  \eqref{9i898u9uer155} holds.
\end{proof}

\begin{theorem}\label{thm:5.3c21}
Let $1\leq k<n$, $0<\varepsilon <1$, $q\leq1$, $q\neq0$, and $-b-q-k=0$. Suppose that $M$ is a $k$-convex star-shaped hypersurface satisfying equation \eqref{eq:5.1} and that $f\in C^{2}(\mathbb{S}^{n})$ is a positive function. Let $\overline{\rho}=(\min_{\mathbb{S}^n}\rho)^{-1}\rho$.
Then there exists a positive constant $C$ that depends on $n$, $k$, $b$, $q$, $|f|_{C^2}$, $\max_{\mathbb{S}^n}f$, and $\min_{\mathbb{S}^n}f$, and is independent of $\varepsilon$, such that
\begin{equation*}
  |D^2 \overline{\rho}| \leq C.
\end{equation*}
\end{theorem}

\begin{proof}
Denote $M_{\overline{\rho}}=\{X\in\mathbb{R}^{n+1}:X=\overline{\rho}(x)x, ~ x \in \mathbb{S}^n\}$. Let $\bar{\kappa}=(\bar{\kappa}_1,\dots,\bar{\kappa}_n)$ be the principal curvatures of $M_{\overline{\rho}}$. By equation \eqref{eq:5.1},
\begin{equation}\label{equation-bar-kappa}
\sigma_k(\bar{\kappa})
= f(\min_{\mathbb{S}^n}\rho)^{-\varepsilon} (\overline{\rho})^{b+2q-\varepsilon}
\bigl(\overline{\rho}^{2}+|D \overline{\rho}|^2\bigr)^{-\frac{q}{2}},
\quad\text{on}\quad\mathbb{S}^n.
\end{equation}
By Theorem \ref{c1225}, there exist positive a constant $C$ depending on $n$, $k$, $q$, $\min_{\mathbb{S}^{n}}f$, $\max_{\mathbb{S}^{n}} f$, and $\max_{\mathbb{S}^{n}}\frac{|D f|}{f}$, but independent of $\varepsilon$, such that
\begin{equation}\label{eq:5.9c1}
1 \leq \overline{\rho} \leq \frac{\max_{\mathbb{S}^n} \rho}{\min_{\mathbb{S}^n} \rho} \leq C,
\end{equation}
and
\begin{equation}\label{eq:5.10c1}
|D \overline{\rho}|
= \frac{\rho}{\min_{\mathbb{S}^n} \rho}\cdot \frac{|D \rho|}{\rho}
\leq \frac{\max_{\mathbb{S}^n} \rho}{\min_{\mathbb{S}^n} \rho}\cdot \frac{|D \rho|}{\rho}
\leq C.
\end{equation}

According to \eqref{eq:5.9c1} and \eqref{eq:5.10c1}, replacing equation (3.7) in \cite{GongTu20262} by equation \eqref{equation-bar-kappa} and then carrying out the same computation as in Theorem 3.5 of \cite{GongTu20262}, it can be proved that for any $\varepsilon\in(0,1)$, there exists a constant $C$ independent of $ \varepsilon$ such that $ |D^2 \overline{\rho}| \leq C $.
\end{proof}

\subsection{Proof of Theorem~\ref{Th-Two} }
We divide the proof  into three steps.\\
\textbf{Step 1:  Existence and strict convexity:} For any small constant $\varepsilon \in (0,1)$, equation \eqref{eq:5.1} has a unique $k$-convex solution $\rho_\varepsilon$ by applying the method used in the proof of Theorem \ref{Th-One}.  Denote $\bar{\rho}_\varepsilon = \dfrac{\rho_\varepsilon}{\min_{\mathbb{S}^n} \rho_\varepsilon}$ and $M_{\bar{\rho}_\varepsilon}=\{X\in\mathbb{R}^{n+1}:X
=\bar{\rho}_\varepsilon(x)x, ~ x \in \mathbb{S}^n\}$.
Let $\bar{\kappa}_\varepsilon$ be the principal curvatures of $M_{\bar{\rho}_\varepsilon}$.
Then
\begin{equation}\label{equation-bar-kappa-epsilon}
  \sigma_k(\bar{\kappa}_\varepsilon)= f(\min_{\mathbb{S}^n}\rho_\varepsilon)^{-\varepsilon} {\bar{\rho}_\varepsilon}^{b+2q-\varepsilon} (\bar{\rho}_\varepsilon^2 + |D \bar{\rho}_\varepsilon|^2)^{-\frac{q}{2}}, \quad\text{on}\quad\mathbb{S}^n.
\end{equation}
It follows from Theorem \ref{thm:5.1} that
\((\min_{\mathbb{S}^n}\rho_\varepsilon)^{-\varepsilon}\) is bounded from
above and below by positive constants independent of \(\varepsilon\).
Using \eqref{eq:5.9c1}, \eqref{eq:5.10c1}, Theorem \ref{thm:5.3c21},
and the Evans--Krylov theorem, one obtains
$|\bar{\rho}_\varepsilon|_{C^{3,\alpha}}\leq C$, where $C$ is independent of $ \varepsilon $. Therefore, there exists a sequence $\{\varepsilon_i\}_{i=1}^{\infty}\subset(0,1)$ such that $\bar{\rho}_{\varepsilon_i}\to \rho$ in $C^{3,\alpha}(\mathbb{S}^n)$ for some $\rho\in C^{3,\alpha}(\mathbb{S}^n)$, and $(\min_{\mathbb{S}^n}\rho_{\varepsilon_i})^{-\varepsilon_i}\to \gamma$ for some positive constant $\gamma$. By equation \eqref{equation-bar-kappa-epsilon}, we obtain that $M=\{X\in\mathbb{R}^{n+1}:X=\rho(x)x, ~ x \in \mathbb{S}^n\}$ satisfies equation \eqref{eq:1.4}.

If $b \leq -k$ and $f$ satisfies \eqref{908998}, then $M_{\bar{\rho}_{\varepsilon_i}}$ is strictly convex for each $i$. Hence $M$ is convex, and by Theorem \ref{crt}, $M$ is strictly convex.
\\
\textbf{Step 2: Uniqueness of the constant $\gamma$:} Assume that there exist two positive constants $\gamma, \tilde{\gamma}$ and two $k$-convex hypersurfaces
$M_\rho=\{X\in\mathbb{R}^{n+1}:X=\rho(x)x, ~ x \in \mathbb{S}^n\}$, $M_{\tilde{\rho}}=\{X\in\mathbb{R}^{n+1}:X=\tilde{\rho}(x)x, ~ x \in \mathbb{S}^n\}$
satisfying
$$\sigma_k(W[\rho])=\gamma f \rho^{b+2q} (\rho^2 + |D \rho|^2)^{-\frac{q}{2}},\quad
\sigma_k(W[\tilde{\rho}])=\tilde{\gamma} f \tilde{\rho}^{b+2q} (\tilde{\rho}^2 + |D \tilde{\rho}|^2)^{-\frac{q}{2}}
, \quad\text{on}\quad\mathbb{S}^n,$$
where $W[\rho]$ and $W[\tilde{\rho}]$ are the Weingarten matrices of $M_\rho$ and $M_{\tilde{\rho}}$, respectively.
Suppose $G = \rho (\tilde{\rho})^{-1}$ attains its maximum at $x_0 \in \mathbb{S}^n$. Then at $x_0$, we have
\begin{equation*}
  0=D \log G=\rho^{-1}D\rho-(\tilde{\rho})^{-1}D\tilde{\rho}
\end{equation*}
and
\begin{equation*}
  0\geq D^2 \log G=D^2 (\log\rho)-D^2(\log\tilde{\rho}).
\end{equation*}
By \eqref{h_ij}, $\rho W[\rho] \geq \tilde{\rho}W[\tilde{\rho}]$.
Therefore, at $x_0$,
\begin{eqnarray*}
\begin{aligned}
\frac{\gamma}{{\tilde{\gamma}}} = \frac{\gamma f(x_0)}{\tilde{\gamma} f(x_0)}
= \frac{\rho^{-b-2q} (\rho^2 + |D \rho|^2)^{\frac{q}{2}} \sigma_k(W[\rho])}{\tilde{\rho}^{-b-2q} (\tilde{\rho}^2 + |D \tilde{\rho}|^2)^{\frac{q}{2}} \sigma_k(W[\tilde{\rho}])}
= \frac{\sigma_k(\rho W[\rho])}{\sigma_k(\tilde{\rho}W[\tilde{\rho}])} \geq 1.
\end{aligned}
\end{eqnarray*}
Similarly, we obtain $\gamma \leq \tilde{\gamma}$. Thus, $\gamma = \tilde{\gamma}$.\\
\textbf{Step 3: Uniqueness of the  solution:} For $i=1,2$, assume that $M_{\rho_i}=\{X\in\mathbb{R}^{n+1}:X=\rho_i(x)x, ~ x \in \mathbb{S}^n\}$
are both $k$-convex and satisfy
\begin{equation*}
  \sigma_k(W[\rho])=\sigma_k(\kappa)=\gamma f \rho^{b+2q} (\rho^2 + |D \rho|^2)^{-\frac{q}{2}}
  , \quad\text{on}\quad\mathbb{S}^n.
\end{equation*}
Here $W[\rho]$ is the Weingarten matrix of $M_{\rho}=\{X\in\mathbb{R}^{n+1}:X=\rho(x)x, ~ x \in \mathbb{S}^n\}$
and its components are denoted by $W_{ij}$, so that $W_{ij}=h^i_j$.
Let
$$\mathcal{G}(\rho)=\sigma_k(W[\rho])\rho^{-b-2q}(\rho^2+|D\rho|^2)^{\frac{q}{2}}.$$
Since $\mathcal{G}$ is invariant under scaling, we may assume $\rho_1 \leq \rho_2$ and $\rho_1(x_0) = \rho_2(x_0)$ for some point $x_0 \in \mathbb{S}^n$. Denote $\rho_t = t\rho_1 + (1-t)\rho_2$ for $0 \leq t \leq 1$. A direct calculation shows that
\begin{equation*}
\begin{aligned}
0 = \mathcal{G}(\rho_2) - \mathcal{G}(\rho_1) &= \int_0^1 -\frac{d}{dt} \mathcal{G}(\rho_t) dt \\
&= \sum_{i,j} \widetilde{a}_{ij}(x)(\rho_1-\rho_2)_{ij} + \sum_i b_i(x)(\rho_1-\rho_2)_i + c(x)(\rho_1-\rho_2),
\end{aligned}
\end{equation*}
where
\begin{equation*}
\begin{aligned}
\widetilde{a}_{ij} =-\int_{0}^{1}\frac{\partial\mathcal{G}(\rho_t)}{\partial(\rho_t)_{ij}}dt= \int_0^1 \rho_t^{-b-2q-1}(\rho_t^2 + |D \rho_t|^2)^{\frac{q-1}{2}}\frac{\partial \sigma_k(W[\rho_t])}{\partial W_{si}}(\delta_{sj}-\frac{(\rho_t)_s
(\rho_t)_j}{\rho_t^2v^2}) \, dt
,
\end{aligned}
\end{equation*}
\begin{equation*}
\begin{aligned}
b_i = -\int_{0}^{1}\frac{\partial\mathcal{G}(\rho_t)}{\partial(\rho_t)_i}dt,\quad c=-\int_{0}^{1}\frac{\partial\mathcal{G}(\rho_t)}{\partial \rho_t}.
\end{aligned}
\end{equation*}
It follows from $(\widetilde{a}_{ij})>0$, $\rho_1-\rho_2\leq 0$, $\rho_1(x_0)-\rho_2(x_0)=0$, and the maximum principle that $\rho_1-\rho_2 \equiv 0$ on $\mathbb{S}^n$. This completes the proof.

\textbf{Conflict of interest statement}
On behalf of all authors, the corresponding author states that there is no conflict of interest.

\textbf{Data availability statement}
No datasets were generated or analysed during the current study.

\textbf{Acknowledgements}
We are grateful to Professor Qiang Tu for his consistent encouragement, support, and many useful suggestions throughout this research.


\begin{thebibliography}{99}

\bibitem{Aleksandrov-1942}
Aleksandrov, A.:
\emph{Existence and uniqueness of a convex surface with a given integral curvature},
In CR (Doklady) Acad. Sci. URSS (NS), 131-134. (1942)

\bibitem{Al1}
Aleksandrov, A.:
\emph{Uniqueness theorems for surfaces in the large},
Vestnik Leningrad. Univ. \textbf{11}, 5-17 (1956); \textbf{12}, 15-44 (1957);
\textbf{13}, 14-26 (1958); \textbf{13}, 27-34 (1958); \textbf{13}, 5-8 (1958);
\textbf{14}, 5-13 (1959); \textbf{15}, 5-13 (1960)







\bibitem{An}
Andrade, F., Barbosa, J., de Lira, J.:
\emph{Closed Weingarten hypersurfaces in warped product manifolds},
Indiana Univ. Math. J. \textbf{58}, 1691-1718 (2009)





\bibitem{Ba1}
Bakelman, I., Kantor, B.:
\emph{Estimates of the solutions of quasilinear elliptic equations that are connected with problems of geometry in the large}(Russian),
 Mat. Sb. (N. S.), \textbf{91}, 336-349, 471 (1973)

\bibitem{Ba2}
Bakelman, I., Kantor, B.:
 \emph{Existence of a hypersurface homeomorphic to the sphere in Euclidean space with a given mean curvature},
 Geometry Topology, \textbf{1}, 3-10 (1974)



 \bibitem{Ba-Li}
Barbosa, J., de Lira, J., Oliker, V.:
\emph{A priori estimates for starshaped compact hypersurfaces with prescribed mth curvature function in space forms},
 In: Nonlinear Problems in Mathematical Physics and Related Topics, vol. I, 35-52. Int. Math. Ser. (N. Y.), 1, Kluwer/Plenum, New York (2002)

\bibitem{BF-2019}
B\"or\"oczky, K., Fodor, F.:
\emph{The $L_p$ dual Minkowski problem for $p>1$ and $q>0$},
J. Differ. Equations \textbf{266}(12), 7980-8033, (2019)


\bibitem{CaffGpMa}
Caffarelli, L., Guan, P., Ma, X.:
\emph{A constant rank theorem for solutions of fully nonlinear elliptic equations},
Commun. Pure Appl. Math. \textbf{60}, 1769-1791 (2007)

\bibitem{Caffarelli-Nirenberg-Spruck-1984}
Caffarelli, L., Nirenberg, L., Spruck, J.:
\emph{The dirichlet problem for nonlinear second‐order elliptic equations I. Monge--Amp\`ere equation},
Commun. Pure Appl. Math. \textbf{37}, 369-402 (1984)

\bibitem{Ca}
Caffarelli, L., Nirenberg, L., Spruck, J.:
\emph{Nonlinear second order elliptic equations, IV. Starshaped compact Weingartenhypersurfaces},
 Current Topics in PDEs. 1-26 (1986)

 \bibitem{CHZ-2019}
Chen, C., Huang, Y., Zhao, Y.:
\emph{Smooth solutions to the $L_p$ dual Minkowski problem},
Math. Ann. \textbf{373}(3), 953-976, (2019)

\bibitem{Chen-Li-Wang-2018}
Chen, D., Li, H., Wang, Z.:
 \emph{Starshaped compact hypersurfaces with prescirbed Weingarten curvature in warped product manifolds},
 Calc. Var. Partial Differ. Equ.
 \textbf{57}, no. 2, Paper No. 42, 26 pp (2018)

 \bibitem{Chen-Li-2021}
Chen, H., Li, Q.:
\emph{The $L_p$ dual Minkowski problem and related parabolic flows},
J. Funct. Anal. \textbf{281}(8), 109139, (2021)



\bibitem{Chou-Wang-2001}
Chou, K., Wang, X.:
\emph{A variational theory of the Hessian equation},
Commun. Pur. Appl. Math. \textbf{54}(9), 1029-1064 (2001)




\bibitem{Fi}
Firey, W.:
\emph{Christoffel problem for general convex bodies},
 Mathematik \textbf{15}, 7-21 (1968)

\bibitem{GongTu20262}
Gong, J., Tu, Q.:
\emph{The existence of  $(\mathbf{p}, k)$-convex hypersurfaces for a class of Hessian quotient type curvature equations}, arXiv:2604.13578v1




 \bibitem{Guan12}
Guan, P., Li, J., Li, Y.:
\emph{Hypersurfaces of Prescribed Curvature Measure},
Duke Math. J. \textbf{161}, 1927-1942 (2012)

\bibitem{Guan-Li-1997}
Guan, P., Li, Y.:
\emph{$C^{1,1}$ estimates for solutions of a problem of Alexandrov},
Commun. Pur. Appl. Math. \textbf{50}(8), 789-811 (1997)

\bibitem{Guan-Lin-Ma-2006}
Guan, P., Lin, C., Ma, X.:
\emph{The Christoffel-Minkowski problem II: Weingarten curvature equations},
Chinese Ann. Math. B \textbf{27}, 595 (2006)

\bibitem{Guan09}
Guan, P., Lin, C., Ma, X.:
\emph{The Existence of Convex Body with Prescribed Curvature Measures},
Int. Math. Res. Not. \textbf{11}, 1947-1975 (2009)

\bibitem{Guan-Ren15}
Guan, P., Ren, C., Wang, Z.:
\emph{Global $C^2$ estimates for convex solutions of curvature equations},
 Commun. Pure Appl. Math. \textbf{68}, 1287-1325 (2015)


\bibitem{Hou-Ma-Wu-2010}
Hou, Z., Ma, X., Wu, D.:
\emph{A second order estimate for complex Hessian equations on a compact Kähler manifold},
Math. Res. Lett. \textbf{17}(3), 547-561 (2010)




\bibitem{HX-2013}
Huang, Y., Xu, L.:
\emph{Two problems related to prescribed curvature measures},
Discrete Contin. Dyn. Syst. \textbf{33}(5), 1975-1986, (2013)

\bibitem{HZ-2018}
Huang, Y., Zhao, Y.:
\emph{On the $L_p$ dual Minkowski problem},
Adv. Math. \textbf{332}, 57-84, (2018)

\bibitem{Hui99}
Huisken, G., Sinestrari, C.:
\emph{Convexity estimates for mean curvature flow and singularities of mean convex surfaces},
Acta Math. \textbf{183}, 45-70, (1999)

\bibitem{Iv1}
Ivochkina, N.:
\emph{Solution of the Dirichlet problem for curvature equations of order $m$},
 Math. USSR-Sbornik. \textbf{67}, 317-339 (1990)

\bibitem{Iv2}
Ivochkina, N.:
\emph{The Dirichlet problem for the equations of curvature of order $m$},
Leningrad Math. J. \textbf{2}, 631-654 (1991)


\bibitem{Jin}
Jin, Q., Li, Y.:
\emph{Starshaped compact hypersurfaces with prescribed $k$-th mean curvature in hyperbolic space},
Discrete Contin. Dyn. Syst. \textbf{15}, 367-377 (2006)



\bibitem{Li-Ol}
Li, Y., Oliker, V.:
\emph{Starshaped compact hypersurfaces with prescribed $m$-th mean curvature in elliptic space},
J. Partial Differ. Equ. \textbf{15}, 68-80 (2002)

\bibitem{Li-Sh}
Li, Q., Sheng, W.:
\emph{Closed hypersurfaces with prescribed Weingarten curvature in Riemannian manifolds},
 Calc. Var. Partial Differ. Equ. \textbf{48}, 41-66 (2013)




\bibitem{LYZ-2018-1}
Lutwak, E., Yang, D., Zhang, G.:
\emph{$L_p$ dual curvature measures},
Adv. Math. \textbf{329}, 85-132, (2018)

\bibitem{Oliker-1983}
Oliker, V.:
\emph{Existence and uniqueness of convex hypersurfaces with prescribed Gaussian curvature in spaces of constant curvature},
Sem. Inst. Mate. Appl. “Giovanni Sansone”, Univ. Studi Firenze (1983)

\bibitem{Ol}
Oliker, V.:
\emph{Hypersurfaces in $\mathbb{R}^{n+1}$ with prescribed Gaussian curvature and related equations of Monge--Amp\`ere type},
Commun. Partial Differ. Equ. \textbf{9}, 807-838 (1984)

\bibitem{Pogorelov-1973}
Pogorelov, A.:
\emph{Extrinsic geometry of convex surfaces},
American Mathematical Soc. (1973)

\bibitem{Ren}
Ren, C., Wang, Z.:
\emph{On the curvature estimates for Hessian equations},
 Am. J. Math. \textbf{141}, 1281-1315 (2019)

 \bibitem{Ren1}
Ren, C., Wang, Z.:
\emph{The global curvature estimates for the $n-2$ Hessian equation}, Calc. Var. Partial Differ. Equ. \textbf{62}, 239 (2023)


\bibitem{Spruck-2005}
Spruck, J.:
\emph{Geometric aspects of the theory of fully nonlinear elliptic equations},
Global theory of minimal surfaces, 283-309 (2005)

\bibitem{SX}
Spruck, J., Xiao, L.: \emph{A note on star-shaped compact hypersurfaces with prescribed scalar curvature in space
forms}, Rev. Mat. Iberoam. \textbf{33}, 547-554 (2017)

\bibitem{Tr}
Treibergs, A., Wei, W.:
\emph{Embedded hypersurfaces with prescribed mean curvature},
J. Differ. Geom. \textbf{18}, 513-521 (1983)





\end{thebibliography}
\end{document}